\documentclass[11pt]{amsart}

\usepackage{amssymb}
\usepackage{amsmath}
\usepackage{amsthm}
\usepackage[mathscr]{eucal}
\DeclareMathOperator{\zone}{zone}

\DeclareMathOperator{\INV}{Inv}

\newcommand{\suchthat}{\,|\,}
\newcommand{\set}[1]{\ensuremath{\{{#1}\}}}

\newcommand{\Inv}[2]{\ensuremath{\INV\langle #1 \suchthat #2 \rangle}}

\newcommand{\sch}{Sch\"ut\-zen\-ber\-ger}
\newcommand{\LETTERS}{\ensuremath{X\cup X^{-1}}}
\newcommand{\WORDS}{\ensuremath{(\LETTERS)^*}}

\newcommand{\si}{\sigma}
\newcommand{\bs}{\bar \sigma}
\newcommand{\ds}{d_{S\Gamma}}
\newcommand{\dd}{d_{\mathcal{D'}}}
\newcommand{\own}{\Omega}
\newcommand{\md}{\mathcal{D'}}
\newcommand{\app}{\sim_{ft}}

\newcommand{\gR}[0]{\ensuremath{\mathscr{R}}}

\newcommand{\pdap}[0]{\ensuremath{\mathscr{P}}}

\newtheorem{defn}{Definition}[section]
\newtheorem{prop}[defn]{Proposition}
\newtheorem{lemma}[defn]{Lemma}
\newtheorem{thm}[defn]{Theorem}
\newtheorem{cor}[defn]{Corollary}

\theoremstyle{remark}

\begin{document}

\title[Inverse monoids presented by a single sparse
  relator]{Decision problems for inverse monoids
  presented by a single sparse relator}

\author[Susan Hermiller]{Susan Hermiller}
\address{Susan Hermiller\\
Department of Mathematics\\
        University of Nebraska\\
         Lincoln NE 68588-0130, USA}
\email{smh@math.unl.edu}

\author[Steven Lindblad]{Steven Lindblad}
\address{Steve Lindblad\\
         Hewitt Associates LLC\\
         45 South 7th Street, Suite 2100, Minneapolis MN 55402, USA}
\email{splindblad@gmail.com}

\author[John Meakin]{John Meakin}
\address{John Meakin\\
Department of Mathematics\\
        University of Nebraska\\
         Lincoln NE 68588-0130, USA}
\email{jmeakin@math.unl.edu}

\begin{abstract}

We study a class of inverse monoids of the form $M=\Inv{X}{w=1}$,
where the single relator $w$ has a combinatorial property that we
call {\bf sparse}.  For a sparse word $w$, we prove that the word
problem for $M$ is decidable. We also show that the set of words
in $\WORDS$ that represent the identity in $M$ is a
deterministic context free
language, and that the set of geodesics in the \sch\ graph of the
identity of $M$ is a regular language.

\end{abstract}

\maketitle

\begin{center}

Dedicated to the memory of Douglas Munn

\end{center}

\section{Introduction}\label{sec:intro}

In a seminal paper in 1974, Douglas Munn \cite{MR50:13328}
introduced the notion of birooted edge labeled trees
(subsequently referred to as ``Munn trees") to solve the word
problem for the free inverse monoid. Munn's work was extended by
Stephen \cite{MR91g:20083} who introduced the notion of \sch\
graphs to study presentations of inverse monoids. The \sch\ graphs
of an inverse monoid presentation are  the strongly connected
components of the Cayley graph of the presentation (or
equivalently the restrictions of the Cayley graph to the
$\gR$-classes of the monoid). From a \sch\ graph for an inverse
monoid presentation, the corresponding \sch\ complex can be
defined as the 2-complex whose 1-skeleton is the \sch\ graph and
whose faces have boundaries labeled by the sides of relations
\cite{MR1971670}.

One-relator inverse monoids of the form $M=\Inv{X}{w=1}$, where $w
\in \WORDS$,  have received some attention in the literature.
Birget, Margolis, and Meakin \cite{MR94j:20057} proved that the
word problem is solvable for inverse monoids of the form
$\Inv{X}{e=1}$, where $e$ is an idempotent in the free inverse
monoid (i.e., reduces to $1$ in the free group). Stephen
\cite{MR93k:20095}  observed that if the inverse monoid
$M=\Inv{X}{w=1}$ is $E$-unitary, then the word problem for $M$ is
decidable if there is an algorithm to decide, for any word $u \in
\WORDS$, whether or not $u=1$ in $M$.  Furthermore, Ivanov,
Margolis, and Meakin \cite{MR2001m:20107} proved that if $w$ is
cyclically reduced, then $M=\Inv{X}{w=1}$ is $E$-unitary.  Thus
the word problem for $M=\Inv{X}{w=1}$, $w$ cyclically reduced, is
reduced to understanding the \sch\ graph of $1$ in $M$. This has
been used to solve the word problem in several special cases (see
for example the paper by Margolis, Meakin and \v Suni\'k
\cite{MMSu}), but the problem remains open in general, even if $w$
is a cyclically reduced word.

 The
present paper is concerned with a class of one-relator inverse
monoids of the form $M=\Inv{X}{w=1}$ where $w \in \WORDS$ satisfies
a combinatorial condition that enables us to understand the
structure of the \sch\ complex corresponding to the identity of $M$.

Let $w=a_0 \cdots a_{n-1}$ with each $a_i$ in $X \cup X^{-1}$. A
{\bf cyclic subword} $q=w(i,j,\epsilon)$ of $w$ is a nonempty word
in $\WORDS$ of length at most $n-1$ of the form $q=a_{{i}}
a_{{i+1}} a_{i+2} \cdots a_{{j-1}}$ if $\epsilon=1$ and
$q=a_{{i-1}}^{-1} a_{{i-2}}^{-1} a_{i-3}^{-1} \cdots a_{{j}}^{-1}$
if $\epsilon=-1$, where $i,j \in {\mathbb Z}/n{\mathbb Z}$. The {\bf
zone} of the cyclic subword $q=w(i,j,\epsilon)$ is the subset of
${\mathbb Z}/n{\mathbb Z}$ given by
$\zone(q):=\{i,i+\epsilon,i+2\epsilon,...,j\}$.

\begin{defn}\label{sparsedefn}
A word $w \in \WORDS$ is {\bf sparse} if $w$ is freely reduced,
$l(w)>1$, and whenever
$(q_k,q_k')=(w(i_k,j_k,\epsilon_k),w(i_k',j_k',\epsilon_k'))$ are
two pairs of cyclic subwords of $w$ satisfying $q_k = q_k'$ in
$\WORDS$, $\zone(q_k) \neq \zone(q_k')$ and $0 \in \zone(q_k')$ for
$k=1,2$, then

\begin{description}

\item[(sparse 1)] $\zone(q_1) \cap \zone(q_2')=\emptyset = \zone(q_1') \cap
\zone(q_2)$, and

\item[(sparse 2)] either $\zone(q_1) \cap \zone(q_2)=\emptyset$ or both
$\epsilon_1 \epsilon_1' = \epsilon_2 \epsilon_2'$ and
   $i_1 - \epsilon_1 \epsilon_1' i_1' = i_2 - \epsilon_2 \epsilon_2' i_2' \mod n$.

\end{description}

\end{defn}

For example one may see easily from this definition that the word
$w = aba^{-1}b^{-1}cdc^{-1}d^{-1}$ and all of its cyclic
conjugates are sparse. However the word $w = aba^{-1}b^{-1}$ is
not sparse. To see this, note that if $q_1 = w(3,2,-1), q_1' =
w(0,1,1), q_2 = w(1,2,1)$ and $q_2' = w(0,3,-1)$, then $q_1 = q_1'
= a$ in $\WORDS$ (where $X = \{a,b,c,d\}$) and $q_2 = q_2' = b$ in
$\WORDS$, but $1 \in \zone(q_1') \cap \zone(q_2)$.

Roughly speaking, if $w$ is a sparse word, then distinct
occurrences of prefixes and suffixes of $w$ that occur
elsewhere as cyclic subwords of $w$ are separated by at least one
letter. This enables us to define an appropriate notion of a dual
graph in the \sch\ complex of $1$ and to prove that this dual
graph is a tree. From this, we can encode the information
contained in the \sch\ complex of $1$ in a pushdown automaton. We
can also show that the faces of this \sch\ complex are of finitely
many types and use this to analyze geodesics and cone types in the
\sch\ graph of $1$. Specifically, we can prove the following
theorems.

\begin{thm}\label{wpsolved}
If $w\in\WORDS$ is sparse, then the word problem for
$M=\Inv{X}{w=1}$ is solvable.
\end{thm}

\begin{thm}\label{cfree}
Let $w$ be sparse and let $M=\Inv{X}{w=1}$. Then:
\begin{enumerate}
\item The language of words equal to $1$ in $M$ is
deterministic context-free.

\item The
language of words related to $1$ by Green's relation $\gR$  in $M$
is deterministic context-free.
\end{enumerate}
\end{thm}

\begin{thm}\label{geodesics}
If $w$ is a sparse word and $M=\Inv{X}{w=1}$, then the language of
geodesics in the \sch\ graph of $1$ for $M$ (i.e. the language of
words labeling geodesic paths starting at $1$ in $S\Gamma(1)$) is
a regular language.
That is, the \sch\ graph of $1$ has finitely many cone types.
\end{thm}

In Section \ref{sec:sparsewords} of the paper we study some
properties of sparse words that enable us to understand how
$n$-gons whose boundaries are labeled by a sparse word may fold
together. Section \ref{sec:sc1} provides information about
sequences of complexes that are used to approximate the \sch\
complex of $1$ for an inverse monoid with sparse relator. Section
\ref{sec:wp} introduces a notion of dual graph to the \sch\
complex of $1$ and this is exploited to provide a proof of Theorem
\ref{wpsolved}. In Section \ref{sec:PDA} we introduce a pushdown
automaton that encodes the information contained in the \sch\
complex of $1$ for a one-relator monoid corresponding to a sparse
word, and we use this to provide a proof of Theorem \ref{cfree}.
We also make use of these results to construct a finite state
automaton that accepts the geodesics in the \sch\ graph of $1$ for
our monoid, and thus provide a proof of Theorem \ref{geodesics}.

We refer the reader to the book of Lawson \cite{MR2000g:20123} for
much of the basic theory of inverse semigroups and to the
paper by Stephen \cite{MR91g:20083} for foundational ideas
and notation about presentations of inverse monoids.

\section{Sparse Words}\label{sec:sparsewords}

Throughout this section, $w=a_0 \cdots a_{n-1}$ will denote a
fixed sparse word in $\WORDS$ as defined in Definition
\ref{sparsedefn} above.

\begin{lemma}\label{cycredlemma}
Every sparse word in $\WORDS$ is cyclically reduced.
\end{lemma}

\begin{proof}
Let $w=a_0 \cdots a_{n-1}$ be a sparse word and suppose that $a_0
= a_{n-1}^{-1} = a$. If we let $q_1 = w(0,-1,-1), q_1' = w(0,1,1),
q_2 = w(0,-1,-1)$ and $q_2' = w(0,1,1)$, then $q_1 = q_1' = q_2 =
q_2' = a$, but $0 \in \zone(q_1) \cap \zone(q_2')$. This contradicts
condition ({\bf sparse 1}) of Definition \ref{sparsedefn}, so $w$
must be cyclically reduced.
\end{proof}

\begin{lemma}\label{primitivelemma}
Every sparse word $w \in \WORDS$ is primitive (i.e. $w$ is not a
proper power in $\WORDS$).
\end{lemma}

\begin{proof}
Suppose that $w = u^m$ in $\WORDS$ for some $m > 1$. The word $u$
has length $l(u) > 0$ since $l(w) > 0$. If we let $q_1 =
w(0,l(u),1) = q_2'$ and $q_1' = w(-l(u),0,1) = q_2$, we again
immediately obtain a contradiction of ({\bf sparse 1}).
\end{proof}

We will build 2-dimensional CW-complexes using information from
the sparse word $w$ to define the attaching maps. To start, let
$P$ be a polygon with $n$ sides; that is, $P$ is a $CW$-complex
with $n$ vertices, $n$ edges and a single 2-cell. We designate a
distinguished vertex $\sigma(P)$ of $P$.  We orient the edges of
$P$ in a clockwise direction, and label the edges of $P$ so that
$w$ is read clockwise from $\sigma(P)$ to $\sigma(P)$ on the
boundary $\partial P$. In addition, we label the vertices of $P$
by the elements of $\mathbb{Z}/n\mathbb{Z}$, starting with 0 at
$\sigma(P)$ and labeling in order also in the clockwise
direction.

We will build finite $2$-complexes iteratively from the $n$-gon
$P$ by successively attaching new copies of $P$ at existing
vertices and applying certain edge foldings. More specifically,
given a finite collection of copies $F_1$,$F_2$,...,$F_m$  of $P$,
first attach the vertex $\sigma(F_2)$ to any vertex of $F_1$ other
than $\sigma(F_1)$.  At the glued vertex $v$, if there are two
edges incident to $v$ with either (1) the same orientation and
edge label, or (2) opposite orientation and edge labels that are
inverse letters in $X \cup X^{-1}$, then we identify those edges
to a single 1-cell (and identify the vertices at the other ends to
a single vertex).  Repeat this successively at all of the vertices
of the complex until no further edge identification according to
rules (1)--(2) can be done, to obtain a new CW-complex with two
2-cells.  Denote the images of $F_1$ and $F_2$ in the quotient by
$\bar F_1$ and $\bar F_2$, respectively, and denote the image of
$\si(F_i)$ by $\bar \si(F_i)$ for $i=1,2$. At the $i$-th step, we
attach $F_i$ to the complex $\bar F_1 \cup \cdots \cup \bar
F_{i-1}$ by identifying $\sigma(F_i)$ with a vertex $v'$ other than
one of the $\bar \si(F_j)$ for $j<i$.  We again glue edges
according to rules (1)--(2) (where the orientation and label of
any edge incident to a face $\bar F_j$ can be considered to be
that inherited from $F_j$), to obtain a quotient CW-complex with
$i$ faces. (Note that at each step, the complex is finite, so this
process must stop.) We say that the face $F_i$ is {\bf folded
onto} $\bar F_1 \cup \cdots \cup \bar F_{i-1}$ {\bf at} $v'$, or
that $F_i$ is {\bf attached} at $v'$.

This process is repeated to create a CW-complex with images $\bar
F_1$,...,$\bar F_m$ of the original polygons as faces.  For any
index $j$ and vertex $v$ in $\bar F_j$, let $i(F_j,v)$ denote
the {\bf index} (or the set of indices) of the vertex (resp. vertices)
in $F_j$ that is sent to $v$ via the canonical map $F_j
\rightarrow \bar F_1 \cup \cdots \cup \bar F_m$.

Note that as a consequence of Lemma \ref{cycredlemma}, the two
edges of a single face $F$ incident to $\si(F)$ cannot be
identified to a single edge in this procedure. The definition of
sparse also implies restrictions on edge gluings in complexes
built from two or three faces, as the following lemmas
demonstrate.  These lemmas will be applied to determine the
structure of the Sch\"utzenberger complex of 1 in Section
\ref{sec:sc1}.

\begin{lemma} [\bf {The two-face lemma}]\label{twofacelemma}

Let $\bar F_1 \cup \bar F_2$ be the CW-complex obtained by folding one face $F_2$
onto another face $F_1$ at a vertex $v
\neq \sigma(F_1)$. Then $\bar \sigma(F_1) \notin \bar{F_1} \cap
\bar{F_2}$.

\end{lemma}

\begin{proof}

Suppose to the contrary that $\bar \sigma(F_1) \in \bar{F_1} \cap
\bar{F_2}$. Since $F_2$ is folded onto the single face $F_1$,
there must be a path in $\bar{F_1} \cap \bar{F_2}$
from $\bar \sigma(F_1)$ to $v = \bar \sigma(F_2)$.
The preimage of this path under the map
$F_1 \rightarrow \bar F_1 \cup \bar F_2$ is a path
in $\partial F_1$ starting at the vertex $\sigma(F_1)$,
and so this path defines a cyclic subword $q_1$ of $w$ starting at vertex 0
when $w$ is viewed as a word labeling $\partial{F_1}$.  Similarly, this
path defines a cyclic subword $q_1'$ of $w$ ending at vertex $0$ when $w$ is
viewed as a word labeling $\partial{F_2}$. The two pairs of cyclic
subwords $(q_1,q_1')$ and $(q_2,q_2'):=(q_1',q_1)$ satisfy $0 \in
\zone(q_1) \cap \zone(q_2')$, contradicting Definition
\ref{sparsedefn}.
\end{proof}

\begin{lemma} [\bf {The three-face lemma}]\label{threefacelemma}
Suppose that the face $F_2$ is folded onto the face $F_1$ with at
least one pair of edges glued, and suppose that face $F_3$ is
folded onto a vertex $v \in \bar{F_1} \cap \bar{F_2}$. Then no
edges are glued via the folding process for $F_3$; that is, no
edge of $F_3$ can be glued to an edge of $\bar{F_1} \cup
\bar{F_2}$, and no two edges of $\bar{F_1} \cup \bar{F_2}$ are
identified.

\end{lemma}

\begin{proof}

By construction, $\bar F_1 \cap \bar F_2$ is a connected non-empty
edge path containing $\bar \sigma(F_2)$ and the vertex $v = \bar
\sigma(F_3)$, so there is a subpath $p_1$ of $\bar F_1 \cap \bar
F_2$ with endpoints $\bar \sigma(F_2)$ and $\bar \sigma(F_3)$.
When viewed as a path in $\partial{F_2}$, $p_1$ determines a
cyclic subword $q_1' = w(i_1', j_1', \epsilon_1')$ such that
$\zone(q_1')$ contains both $0=i(F_2,\bar \si(F_2))$ and the index $i(F_{2},v)$
of the vertex corresponding to $v$. When viewed as a path in
$\partial{F_1}$, $p_1$ determines a cyclic subword $q_1 =
w(i_1,j_1,\epsilon_1)$ such that $\zone(q_1)$ contains
$i(F_{1},\bar \sigma(F_2))$ and $i(F_{1},v)$.

Suppose that some edge of $F_3$ is glued onto an edge of $\bar F_1
\cup \bar F_2$.

\underline{Case 1}. $F_3$ folds onto an edge of $\bar F_1$. Then
there is a non-trivial path $p_2$ in $\bar F_1 \cap \bar F_3$ with
endpoint $v = \bar \sigma(F_3)$. When viewed as a path in
$\partial{F_3}$, $p_2$ determines a cyclic subword $q_2' =
w(i_2',j_2',\epsilon_2')$ with $0 \in \zone(q_2')$. When viewed as
a path in $\partial{F_1}$, $p_2$ determines a cyclic word $q_2 =
w(i_2,j_2,\epsilon_2)$ such that $i(F_{1},v) \in \zone(q_2)$. Then
$i(F_{1},v) \in \zone(q_1)
\cap \zone(q_2) \neq \emptyset$. But $i_1-\epsilon_1\epsilon_1'i_1'
= i(F_1,\bar \sigma(F_2))$ and $i_2-\epsilon_2\epsilon_2'i_2' =
i(F_1,v)$, so $i_1-\epsilon_1\epsilon_1'i_1' \neq
i_2-\epsilon_2\epsilon_2'i_2'$, contradicting condition ({\bf
sparse 2}) of Definition \ref{sparsedefn}. Thus Case 1 cannot
occur.

\underline{Case 2}. $F_3$ folds onto an edge of $\bar F_2$.
Then there is a non-trivial path $p_3$ in $\bar
F_2 \cap \bar F_3$ with endpoint $v = \bar \sigma(F_3)$. When
viewed as a path in $\partial{F_3}$, $p_3$ determines a cyclic
subword $q_3' = w(i_3',j_3',\epsilon_3')$ with $0 \in \zone(q_3')$.
When viewed as a path in $\partial{F_2}$, $p_3$ determines a
cyclic subword $q_3 = w(i_3,j_3,\epsilon_3)$ with $i(F_2,v) \in
\zone(q_3)$. In this case, $i(F_2,v) \in \zone(q_1') \cap \zone(q_3)
\neq \emptyset$, so condition ({\bf sparse 1}) fails, a
contradiction.

Since no edge of $F_3$ is folded onto any edge of $\bar{F_1} \cup
\bar{F_2}$, no additional edge folding can occur in $\bar{F_1}
\cup \bar{F_2}$.
\end{proof}

\section{The Sch\"utzenberger complex $SC(1)$}\label{sec:sc1}

Throughout this section, $w$ will denote a fixed sparse word and
$M=\Inv{X}{w=1}$. We recall that the \sch\ graph of $1$ for this
presentation is the restriction of the Cayley graph of $M$ to the
$\gR$-class of $1$. We denote this graph by $S\Gamma(1)$: its
vertices are the elements $s \in M$ such that $ss^{-1} = 1$ in $M$
and there is an edge  labeled by $x \in X \cup X^{-1}$ from $s$
to $t$ if $ss^{-1} = tt^{-1} = 1$ and $sx = t$ in $M$. We denote
this edge by $(s,x,t)$. Its inverse edge is the edge
$(t,x^{-1},s)$ in $S\Gamma(1)$, where we interpret $(x^{-1})^{-1}
= x$, and this inverse pair is interpreted as a
single topological edge.
The \sch\ complex of $1$ is the complex $SC(1)$
obtained from $S\Gamma(1)$ by adding a face with boundary label
$w$ for each closed path labeled by $w$ in $S\Gamma(1)$.
Stephen's iterative construction of a sequence of approximations
of the \sch\ graph $S\Gamma(1)$ may easily be adapted to yield a
sequence of approximations of the \sch\ complex $SC(1)$. In
particular, we may construct such a sequence of complexes in the
following way.

Start with a trivial complex $S_0$ consisting of one vertex $v_0$
and no edges or faces. Take a copy $F_1$ of the $n$-gon $P$,
identify its start vertex $\sigma(P)$ with $v_0$, and denote this
complex by $S_1$. As in Section \ref{sec:sparsewords}, we build a
sequence of complexes $S_1 = \bar F_1$, $S_2 = \bar F_1 \cup \bar
F_2$, $S_3 = \bar F_1 \cup \bar F_2 \cup \bar F_3$, ... by
successively folding faces $F_i$ onto $\bar F_1 \cup \bar F_2
\cup \cdots \cup \bar F_{i-1}$ at vertices $v_{i-1} \in S_{i-1}$ at which no
face has yet been attached, in such a way that $d(v_0,v_{i-1})$ is
as small as possible, where $d$ is the path metric in $S_{i-1}$.
Lemma \ref{primitivelemma} guarantees that such a $v_{i-1}$ exists. To
see this, note that if no such $v_{i-1}$ exists, then $S\Gamma(1) =
S_{i-1}$, so $S\Gamma(1)$ is finite. Thus if $x$ is the first
letter in $w$, since $x^j$ labels a path in $S\Gamma(1)$ for each
$j > 0$ we see that $x$ is a torsion element in $M$ (i.e. $x^{j} =
x^{k}$ for some $k \neq j$). It follows that $x$ must be a torsion
element of $G=Gp\langle X | w=1\rangle$, but $G$ is torsion free
if $w$ is primitive.

A sequence of complexes obtained in the above manner is referred
to as a {\bf \sch\ approximation sequence}. Since $v_i=\bar
\si(F_{i+1})$ is chosen so as to minimize the distance from $v_0$,
we can see that every vertex of $S_i$ is the start vertex of some
face in $S_{i+j}$ for some $j$. From the results of Stephen
\cite{MR91g:20083}, the corresponding sequence of $1$-skeleta of a
\sch\ approximation sequence has a direct limit  that is
independent of the choice of the vertices $v_i$, and this direct
limit is $S\Gamma(1)$. By an argument similar to the formal
category theoretical argument in \cite{MR91g:20083} used to show
this, it follows that the \sch\ approximation sequence of
complexes has a direct limit, and since the approximation sequence
attaches faces whenever a closed path labeled by $w$ is attached,
the limit of the \sch\ approximation sequence is the \sch\ complex
$SC(1)$.

\begin{thm}\label{thm:sc1}

Let $S_0,S_1,S_2, \ldots$ be any \sch\ approximation sequence for
$SC(1)$ corresponding to a sparse word $w$. Then for all $m \geq
0$ and for all distinct faces $\bar F_i,\bar F_j,\bar F_k,\bar F_l$ in $S_m$

\begin{enumerate}

\item The natural map $F_i \to \bar F_i$ is an embedding of $F_i$
into $S_m$.

\item If $\bar F_i \cap \bar F_j \neq \emptyset$, then $\bar F_i
\cap \bar F_j$ is a connected path
such that either
$\bar\sigma(F_i) \in \bar F_j$ with $\bar\sigma(F_j) \notin \bar F_i$,
or
$\bar\sigma(F_j) \in \bar F_i$ with $\bar\sigma(F_i) \notin \bar F_j$.

\item If $\bar F_i \cap \bar F_j \cap \bar F_k \neq \emptyset$,
then there exists $r \in \{i,j,k\}$ with $\bar F_i \cap \bar F_j
\cap \bar F_k = \bar \sigma(F_r)$ and $\bar F_r$ shares no other
vertices with the other two faces.

\item $\bar F_i \cap \bar F_j \cap \bar F_k \cap \bar F_l =
\emptyset$.

\item The natural map from $S_{m-1}$ to $S_m$ is an embedding.

\end{enumerate}

\end{thm}

\begin{proof}

The proof proceeds by induction on $m$. The result is clear if $m$
is 0 or $1$. Suppose that the result is true for approximation
sequences of length $m-1$. Let $v$ be the vertex of $S_{m-1}$ at
which $F_m$ is attached to $S_{m-1}$.  From part (4) of the
induction assumption, at most three faces contain
the point $v$.

\underline{Case 1}. Suppose that $v$ is on the boundary of three
faces in $S_{m-1}$. Then by part (3) of the induction assumption,
one of these faces $\bar F$ satisfies $v = \bar \sigma(F)$. But
then the algorithm for constructing the \sch\ approximation
sequence would not attach $F_m$ at $v$ also. Hence Case 1 cannot
occur.

\underline{Case 2}. Suppose that $v$ is on the boundary of exactly
two faces $\bar F_i$ and $\bar F_j$ in $S_{i-1}$. By part (2) of
the induction hypothesis, we may assume without loss of generality
that $\bar \sigma(F_i) \in \bar F_j$ and again by this induction
hypothesis there is a non-trivial path in $\bar F_i \cap \bar F_j$
from $v$ to $\bar \sigma(F_i)$. Then by the three-face lemma, no
edge of $F_m$ is glued onto any edge of $\bar F_i \cup \bar F_j$
at $v$, and hence no edge of $F_m$ is glued onto any edge of
$S_{m-1}$ at all. Hence properties (1)--(5) of the statement of
the theorem hold for $S_m$.

\underline{Case 3}. Suppose that $v$ is on the boundary of exactly
one face $\bar F_i$ of $S_{m-1}$. Consider the complex $\hat{S_m}$
obtained from $S_{m-1}$ and $F_m$ by just gluing edges of $F_m$
and $\bar F_i$ starting from $v$, and no additional edge foldings.
Then $\bar F_i \cap \hat F_m$ is a connected path. If there exists
a vertex ${v'}$ in $\bar F_i \cap \hat{F_m}$ with ${v'} \neq
v=\hat \si(F_m)$, the two-face lemma says that ${v'} \neq \bar
\sigma(F_i)$ also. In this case the three-face lemma then says
that any other face incident to ${v'}$ cannot contain an edge that
can be identified with an edge of either $\bar F_i$ or $\hat{F_m}$
in a further folding process. Thus in any case no further edges
can be glued, and $\hat{S_m} = S_m$. Hence properties (1)--(5) of
the statement of the theorem hold for $S_m$.
\end{proof}

Using part (5) of Theorem \ref{thm:sc1}, we may consider $S_0
\subset S_1 \subset S_2 \subset S_3 \subset \cdots$, and so $SC(1)
= \cup_{m=0}^{\infty}S_m$ for any \sch\ approximation sequence
constructed as above.  Hence the corollary below follows
immediately.

\begin{cor}\label{cor:sc1}
Properties (1)--(5) of Theorem \ref{thm:sc1} hold with $S_m$
replaced by $SC(1)$.
\end{cor}

For every \sch\ approximation sequence, there is a vertex $v_0$
which is the unique vertex incident to only one face in the
direct limit, and so there is a unique vertex in $SC(1)$, which
we will also call $v_0$, that is incident to only one face,
which we will refer to throughout as the face $F_1$.  For any
face $A$ of $SC(1)$, the sparse property of $w$ implies
that there is only one vertex in $\partial A$ that can
be the start vertex $\bs(A)$, and only one possible orientation
starting from this vertex in which the word $w$ labels the
boundary path.

For distinct faces $A$ and $B$ of $SC(1)$, we define $A<B$ if the face
$A$ must be attached before the face $B$ in every \sch\ approximation
sequence.   The corresponding partial ordering $\le$
is the {\bf face ordering} on the faces of $SC(1)$.  This partial
ordering is well-founded, and the face $F_1$ is a minimal element.

\begin{cor} [\bf {Order Corollary}]\label{ordercor}
Let $v$ be a vertex of $SC(1)$, and let $B$ be the face
with $v=\bar \sigma(B)$.
\begin{enumerate}
\item If $v$ is incident to exactly one other face $A$, then $A<B$.
\item If $v$ is incident to two other faces $A$ and $C$ with
$\bar \sigma(C) \in A$, then $A<B$ and $A<C$.
\item If $v$ is incident to a face $A$ and $A \cap B$ contains
at least one edge, then $A<B$.
\end{enumerate}
\end{cor}

\begin{proof}
Let $S_0,S_1,S_2, \ldots$ be any \sch\ approximation sequence for
$SC(1)$ corresponding to a sparse word $w$, with face $F_i$ attached
to $S_{i-1}$  in the construction of $S_i$, as above.
In the case that $v$ is incident only to faces $A=\bar F_j$ and $B=\bar F_k$,
the vertex $v$ must exist in a complex $S_i$ before $B$ can
be attached, and so we must have $j<k$.

In the case that $v$ is also incident to a third face $C=\bar F_l$ with
$\bar \sigma(C) \in A$,
then Theorem \ref{thm:sc1} says that $A \cap C$ contains a connected
non-empty edge path from $\bar \sigma(C)$ to $v$, and so
at the vertex $\bar \sigma(C)$, an edge of $C$ is glued to an edge of $A$.
Again applying Theorem \ref{thm:sc1}, no face other than $A$ and
$C$ can be incident to $\bar \sigma(C)$ in any of the $S_i$.
Then as in the paragraph above, we have $j<l$.  Now the face
$F_k$ can be attached at $v$ only after $v$ has been built in the
sequence, and hence only after at least one of $F_j$, $F_l$
has been attached.  Therefore $j<k$ also.

Finally, if $v \in A$ and $A \cap B$ contains at least one edge,
then Theorem \ref{thm:sc1} says that no other face can be
incident to $v$, and so the first paragraph of this proof applies.
\end{proof}

\section{The dual graph and the word problem  }\label{sec:wp}

In this section we define a notion of a dual graph of  the \sch\
complex $SC(1)$ for  an inverse monoid $M=\Inv{X}{w=1}$
corresponding to a sparse word $w$. We show that this dual graph
is a tree and we make use of this to provide a solution to the
word problem for $M$.

\begin{defn}\label{dualgrdefn}
Let $w$ be a sparse word. The {\em dual graph} of $SC(1)$ for
$M=\Inv{X}{w=1}$ is the directed graph $\mathcal{D}$ with
\begin{itemize}
\item vertex set $V(\mathcal{D})$ given by the set of faces of $SC(1)$, and
\item set $E(\mathcal{D})$ of directed edges $(A,B)$ (oriented from
$A$ to $B$)
for $A,B \in V(\mathcal{D})$ satisfying
$A<B$ in the face ordering and $A \cap B \neq \emptyset$ in $SC(1)$.
\end{itemize}
\end{defn}

As a consequence of Corollaries \ref{cor:sc1} and \ref{ordercor}, the
definition of $E(\mathcal{D})$ can also be phrased purely in
terms of the combinatorial properties of  $SC(1)$, namely
$(A,B)$ is a directed edge in $\mathcal{D}$ if and only if
$A\neq B$, $\bar \si(B) \in A$,
and whenever $C \in V(\mathcal{D})$
with $\bar \si(B) \in C$ then $\bar \sigma(C)\in A$.

\begin{prop}\label{dualgraphisatree}
Let $w$ be a sparse word and $M=\Inv{X}{w=1}$. Then the dual graph
$\mathcal{D}$ of $SC(1)$ is a directed, rooted, infinite tree
(with root $F_1$) in which each vertex has at most $l(w)-1$
children.
\end{prop}
\begin{proof}
Recall that the face $F_1$ is the only face of $SC(1)$
containing the unique vertex $v_0$ of $SC(1)$ incident to only
one face.  Let $A \neq F_1$ be any other face in $SC(1)$,
and assume by Noetherian induction that for all faces $B<A$
with respect to the well-founded face ordering, there is a directed
edge path in $\mathcal{D}$ from $F_1$ to $B$.
From Corollary \ref{cor:sc1}, there are either 2 or 3 faces
incident to the vertex $\bar \sigma(A)$ in $SC(1)$, including
$A$.

If there is only one other face $B$ incident to $\bar \sigma(A)$,
then the Order Corollary \ref{ordercor} implies that $B<A$.  Since
$\bar \sigma(B) \in A \cap B \neq \emptyset$, then $(B,A) \in E(\mathcal{D})$.
The concatenation of the path from $F_1$ to $B$ from the
induction assumption with this edge $(B,A)$ then gives
a directed edge path in $\mathcal{D}$ from $F_1$ to $A$.

On the other hand, if there are two other faces $B$ and $C$
incident to $\bar \sigma(A)$, then Corollary \ref{cor:sc1}
says that one of these faces contains the $\bar \sigma$ vertex
of the other; without loss of generality, suppose that
$\bar \sigma(C) \in B$.  Then the Order Corollary \ref{ordercor}
again implies that $B<A$, and as in the previous paragraph
we obtain a directed path in $\mathcal{D}$ from $F_1$ to $A$.
Hence $\mathcal{D}$ is connected.

Suppose that $\mathcal{D}$ is not a tree.  Then there
is an undirected circuit in this graph.

Suppose that two edges of this
circuit have a common target; that is, suppose that there are
edges $(A,C), (B,C) \in E(\mathcal{D})$
with $A \neq B$.  Using the combinatorial description of
$E(\mathcal{D})$ above, then
$\bar \si(C) \in A \cap B$.  From Corollary
\ref{cor:sc1} part (2),
$A \cap B$ is a path containing one of $\bar \si(A)$
or $\bar \si(B)$ but not both.
This contradicts the existence of one of the edges $(A,C), (B,C)$,
and so the circuit must also be a directed circuit.

The consecutive vertices $A_1,A_2,...,A_k$ following
the directed edges in this circuit must
then satisfy $A_1<A_2< \cdots A_k<A_1$ in the face ordering,
which is again a contradiction.  Hence $\mathcal{D}$ is a
directed tree with root $F_1$.

Since each face $A$ of $SC(1)$ has $l(w)-1$ vertices other than its
vertex $\bar\si(A)$, there are at most $l(w)-1$ directed
edges in $\mathcal{D}$ with source
vertex $A$.  In
addition, as remarked in Section \ref{sec:sc1}, the fact that $w$ is primitive
guarantees $SC(1)$ is infinite. Therefore the tree $\mathcal{D}$ must be
infinite.
\end{proof}

To simplify notation later, it will be helpful to
consider a slight modification of $\mathcal{D}$.
The {\bf augmented dual graph} ${\bf \mathcal{D}'}$ is
obtained from $\mathcal{D}$ by
adding an additional vertex $v_0$ to $\mathcal{D}$ and  an
additional directed edge from $v_0$ to $F_1$.
Then $\mathcal{D}'$ is a directed rooted tree with root $v_0$.
Using standard language for rooted trees, if $(A,B)$ is a
directed edge in $\mathcal{D'}$, we call $A$ the parent of
$B$, and $B$ a child of $A$.

Define a map $\Omega:V(SC(1)) \rightarrow V(\mathcal{D'})$
as follows.
For each vertex $v \neq v_0$ in $SC(1)$, let
$\Omega(v)$ be the unique face of $SC(1)$
that is closest to $v_0$ in
$\mathcal{D'}$ from among the faces that are incident to $v$, and
let $\Omega(v_0) := v_0$.

For any face $A$ of $SC(1)$, then $\Omega(\bar\si(A))$ is the
parent of $A$, and so $\Omega(\bar\si(A))<A$;
i.e., $\Omega(\bar\si(A))$ must be attached before
$A$ in any \sch\ approximation.
By Corollaries \ref{cor:sc1} and \ref{ordercor},
in the folding process edges
of $A$ can be glued to edges of $\Omega(\bar\si(A))$
but not to edges of any other face, and
the glued edges are a connected path.
Recall that the boundary $\partial A$ of the polygon $A$
is labeled by the word $w$, when read starting at the
vertex $\si(A)$ in the clockwise direction.  The connected
set $\gamma(A):=A \cap \Omega(\bar\si(A))$, then, can be regarded
as the image of a (``gluing'') path (which we will also
call $\gamma(A)$) going clockwise around $\partial A$
from the (``reverse'') vertex $\rho(A)$
to the (``forward'')
vertex $\phi(A)$.
Note that if no edges are glued when $A$ is attached to
its parent $\Omega(\bar\si(A))$, then $\rho(A)=\bar\si(A)=\phi(A)$
and $\gamma(A)$ is this point.

\begin{lemma}\label{halflemma}

Let $A$ be a face of the complex $SC(1)$ for a sparse word $w$.
\begin{enumerate}
\item The lengths $l(w)$ and $l(\gamma(A))$ satisfy
$l(\gamma(A)) \le \frac{1}{2}l(w)-1$.
\item If $v$ is a vertex in $\gamma(A)$, then
$\own(v)=\own(\bs(A))$.
\item If $v$ is a vertex in $\partial A \setminus \gamma(A)$,
then $\own(v)=A$.
\end{enumerate}
\end{lemma}

\begin{proof}
The path $\gamma(A)$ determines a cyclic subword $q'$ of $w$
when viewed as a
path in $\partial{A}$, and determines a cyclic subword $q$ when viewed as a
path in the parent $\partial{\Omega(\bar\si(A))}$ of $A$.
Since $w$ is sparse we must have $\zone(q')
\cap \zone(q) = \emptyset$, (take $(q_1,q_1') = (q_2,q_2') =
(q,q')$ in Definition \ref{sparsedefn}), and so
there must also be at least one edge between the endpoints
of these cyclic subwords on both sides.
Then $l(w) \ge 2 l(\gamma(A)) +2$.

If $v$ is a vertex in  $\gamma(A)$, then by
definition of the set $\gamma$, the point $v$ is
also in the parent $\own(\bs(A))$ of $A$.
If there is a third face $C$ incident
to $v$ in $SC(1)$, then by the order corollary and the definition
of $\mathcal{D'}$, the face $\own(\bs(A))$ is also the parent of $C$.

For a vertex $v$ in $\partial A \setminus \gamma(A)$,
then $v=\bs(B)$ for another face $B$ of $SC(1)$.  If there is no
other face incident to $v$, the order corollary then says $A<B$.
If $C$ is a third face incident to $v$, then Corollary \ref{cor:sc1}
says that $A$ and $C$ must share at least one edge in common,
and either $\bs(A)$ is in $C$, or $\bs(C)$ is in $A$.
The order corollary then says that the face among $A$ and $C$
that contains the start vertex $\bs$ of the other is the parent
of the pair.  However,
since $v \notin \own(\bs(A))$, we must have
$C \neq \own(\bs(A))$, and hence
$\bs(C) \in A$
and $A$ is the parent of both $B$ and $C$.
\end{proof}

Let the 1-skeleton $S\Gamma(1)$ of the 2-complex $SC(1)$
have the path metric $\ds$, and let
the augmented dual graph have path metric $\dd$.  The following
theorem shows that geodesics in these metric spaces are closely related.

\begin{thm}[\bf {Geodesic Theorem}]\label{geodesicprop}
Let $p$ be any geodesic edge path in $S\Gamma(1)$ from $v_0$ to a
vertex $v$.  Let $v_0,v_1,...,v_k=v$ be the successive vertices in
the path $p$.  Then for all $i$, either
$\Omega(v_i)=\own(v_{i+1})$ or $\Omega(v_i)$ is the parent of
$\own(v_{i+1})$ in $\mathcal{D'}$, and the edge from $v_i$ to
$v_{i+1}$ is contained in $\own(v_{i+1})$. Moreover, whenever
$\Omega(v_i) < \own(v_{i+1})$, then $v_i \in
\{\rho(\own(v_{i+1})), \phi(\own(v_{i+1}))\}$.
\end{thm}

\begin{proof}
We prove this by induction on the length $k$ of the edge
path $p$.  If $k=0$, then $p$ is the constant path at $v_0$
in $SC(1)$, and there is no other vertex.  If $k=1$, then
$p$ follows a single edge from $v_0$ to $v_1 \neq v_0$.
Then $\own(v_0)=v_0$ is the parent of $\own(v_1)=F_1$ in $\mathcal{D'}$.

Suppose that $k \ge 2$.  The prefix $\hat p$ of the path $p$ with
vertices $v_0,...,v_{k-1}$ is also a geodesic path in
$S\Gamma(1)$, and so by induction the conditions on the pair
$\own(v_i),\own(v_{i+1})$ in the theorem hold for all $0 \le i \le
k-2$. The vertex $v_{k-1} \neq v_0$, so Corollary \ref{cor:sc1}
says that there are at least two faces $A:=\own(v_{k-1})$ and $B$
with $\bs(B)=v_{k-1}$, and possibly a third face $C$, incident to
the vertex $v$ in $SC(1)$. By definition of $\own$ and the Order
Corollary, we have $A<B$ and $A<C$.  The edge $e$ from $v_{k-1}$
to $v_k$ must be contained in at least one of these faces.

\underline{Case 1.}  Suppose that $e$ is contained in $A$.
If $v_k$ is in the path $\gamma(A)$, then
Lemma \ref{halflemma} implies that $\own(v_k) = \own(\bs(A))$,
but since $\own(v_{k-1})=A$, the same lemma implies that
$v_{k-1}$ is not in $\gamma(A)$.  Then $v_k$ must be
one of the endpoints $\rho(A)$, $\phi(A)$ of $\gamma(A)$.
By induction, the prefix $\hat p$ of $p$ traversed one of
these endpoints, and since $p$ is a geodesic, $\hat p$ must
have traversed the endpoint $v'$ of $\gamma(A)$ that is not $v_k$.
However, this implies that a suffix of $p$ is
a geodesic in $\partial A$ from $v'$ to $v_k$
that goes through the point $v_{k-1}$ not in $\gamma(A)$.
This contradicts Lemma \ref{halflemma}(1), and so $v_k$
must lie in $\partial A \setminus \gamma(A)$.
Lemma \ref{halflemma}(3) then implies that $\own(v_k)=A=\own(v_{k-1})$.

\underline{Case 2.} Suppose that $e$ is contained in a child
$E$ of the face $A$, but not in $A$.  That is, $E$ is one of
the faces $B$ or $C$.  In this case,
since $v_k$ is not contained in $A \cap E$, then Lemma \ref{halflemma}(3)
says that $\own(v_k)$ is $E$, and we have that
$\own(v_{k-1})=A$ is the parent of $\own(v_k)$.
Moreover, since $v_{k-1}$ is in $A \cap E$
but $v_k$ is not,
we have that $v_{k-1}$ is one of the endpoints
$\rho(E)$, $\phi(E)$ of $\gamma(E)$.
\end{proof}

We can now provide a solution to the word problem for $M$.

\medskip

\noindent {\bf Proof of Theorem \ref{wpsolved}}.
As noted in Section \ref{sec:intro}, it is
sufficient to prove that there is an
algorithm that takes a word $u\in\WORDS$ as input, and outputs
whether or not $u=1$ in $M$.  Given a sparse word $w$, the
following procedure is such an algorithm for $M=\Inv{X}{w=1}$.

Let $L:=l(u)$ be the length of the
word $u$.
The algorithm follows the construction of a
\sch\ approximation sequence as described at the beginning of
Section \ref{sec:sc1}, attaching a face at each step
to a vertex whose distance to $v_0$ in the approximation complex
is minimal from
among all of those vertices that are not yet the
start vertex ($\bs$) of a face.
Continue this process until
the next vertex at
which a face is to be attached has distance
$L \cdot l(w)+1$ from $v_0$; the
process stops at this time, with an approximation complex $S$.
Since each complex in this
sequence is locally finite, this process is finite.

From Theorem \ref{thm:sc1} we know that $S$ embeds in $SC(1)$.
From the Geodesic Theorem \ref{geodesicprop}, we have that
for each vertex $v$ in $S$,
any geodesic path $p$ in $S\Gamma(1)$ from $v_0$ to $v$ is contained
in the union of the
the faces labeling
vertices of the geodesic in $\md$ from $v_0$ to $\own(v)$.  By the
definition of the map $\own$,
these are the faces  that must be
constructed in the \sch\ approximation sequence before
the face $\own(v)$, together with the face
$\own(v)$ which
must be the first face containing $v$ constructed in the sequence.
Hence all of these faces are also in $S$, as is the path $p$.
Therefore the
path metric $d_S$ in the 1-skeleton of $S$ is the same as the metric
inherited from $S\Gamma(1)$.

We claim that every face $A$ of $SC(1)$ with $\dd(v_0,A) \le L$
lies in $S$.  Suppose not; that is, suppose that there is
a face $A$ with $\dd(v_0,A) \le L$ and $A$ not in $S$, and
choose $A$ to have minimal distance from $v_0$ in $\md$ among
all such faces.  Then the parent $\own(\bs(A))$ of $A$ satisfies
$\dd(v_0,\own(\bs(A)))=\dd(v_0,A)-1$, and so $\own(\bs(A))$ lies
in $S$.  But the previous paragraph and
Theorem \ref{geodesicprop} imply that
$d_S(v_0,\bs(A)) = \ds(v_0,\bs(A)) < L \cdot l(w)$,
and so $S$ has a vertex $\bs(A)$ within
$L \cdot l(w)$ of $v_0$ that is not the start vertex of
an attached face, giving the required contradiction.

Now let $v'$ be any vertex of $SC(1)$ with $\ds(v_0,v')
\leq L$.  It follows from Theorem \ref{geodesicprop}
that $\dd(v_0, \own(v')) \leq L$ also, and so by the
previous paragraph, $\own(v')$, and hence also $v'$,
is in the finite complex $S$.
Putting these results together, we have that $u$ labels a path from
$v_0$ to $v_0$ in $SC(1)$ if and only if $u$ labels a path
from $v_0$ to $v_0$ in $S$.
Since $u=1$ in $M$ if and only if $u$ labels a path in $SC(1)$ from
$v_0$ to $v_0$, the algorithm outputs $u=1$ if $u$ labels a path
from $v_0$ to $v_0$ in $S$, and outputs $u\neq 1$ in $M$
otherwise. \hfill $\Box$

\section{Languages of geodesics and words representing $1$}\label{sec:PDA}

Throughout this section, $w$ is a sparse word and
$M=\Inv{X}{w=1}$.

\begin{lemma}\label{furthest}
  Let $A$ be any face in $SC(1)$.  There is
a unique point $x_A$ in $\partial A \setminus \gamma(A)$
satisfying $\ds(v_0,x_A) \ge \ds(v_0,y)$
for all $y \in \partial A$.
\end{lemma}

\begin{proof}
First consider points in the set
$T:=\partial A \setminus \gamma(A)$.
Lemma \ref{halflemma}(3) and the Geodesic Theorem \ref{geodesicprop}
imply that every
geodesic from $v_0$ to a point $y$ in $T$ must
traverse one of the points $\rho(A),\phi(A)$, and then
follow edges in the path along $T$ to $y$.
Let $a:=\ds(v_0,\rho(A))$, $b:=\ds(v_0,\phi(A))$, and $q:=l(\gamma(A))$,
and let $p$ be the length of the edge path in $T$ from
$\rho(A)$ to $\phi(A)$.
The triangle inequality together with Lemma \ref{halflemma}(1)
give $|b-a| \le q \le p-1$.
Let $x$ be the point in $T$ that is a distance
$\frac{1}{2}(p+(b-a))<p$ from the endpoint $\rho(A)$; then
$x$ is a distance $\frac{1}{2}(p+(a-b))$ along $T$
from $\phi(A)$.  Now the concatenation of
a geodesic path from $v_0$ to $\rho(A)$ followed
by the geodesic in $T$ from $\rho(A)$ to $x$ has
the same length $\frac{1}{2}(p+a+b)$ as the concatenation of
a geodesic path from $v_0$ to $\phi(A)$ followed
by the geodesic in $T$ from $\phi(A)$ to $x$, and
hence both of these concatenations
are geodesics from $v_0$ to $x$.
Since every other point
$y \in T$ lies on one of these paths,
we have $\ds(v_0,x) > \ds(v_0,y)$.

Similarly, let $z$ be the point in $\gamma(A)$
that is a distance $\frac{1}{2}(q+(b-a)) \le q$
from the endpoint $\rho(A)$ along the path $\gamma(A)$,
and hence a distance $\frac{1}{2}(q+(a-b))$ from $\phi(A)$.
The concatenation of a geodesic from $v_0$ to either
$\rho(A)$ or $\phi(A)$,
together with the geodesic along $\gamma(A)$ from that
endpoint to $z$, has length $\frac{1}{2}(q+a+b)$,
and every point $y$ in $\gamma(A)$ lies on one of these
path concatenations.  Hence for all $y \in \gamma(A)$,
we also have $\ds(v_0,x) = \frac{1}{2}(p+a+b)>\frac{1}{2}(q+a+b)
 \ge \ds(v_0,y)$.
\end{proof}

For a face $A$ of $SC(1)$, choose $\mathbb{Z}$ representatives
$\hat i(A,\rho(A))$ and $\hat i(A,\phi(A))$ of the
indices $i(A,\rho(A))$ and $i(A,\phi(A))$ from $\mathbb{Z}/n\mathbb{Z}$,
respectively,
satisfying $0 \le \hat i(A,\phi(A)) < \hat i(A,\rho(A)) \le n =l(w)$.
Similarly, for each vertex $v$ in $\partial A \setminus \gamma(A)$,
let $\hat i(A,v)$ be the representative of $i(A,v)$
satisfying $0 < \hat i(A,v) <  n$; from Lemma \ref{halflemma},
then $\hat i(A,\phi(A)) < \hat i(A,v) < \hat i(A,\rho(A))$.

Define $k_A := \frac{1}{2}[\hat i(A,\rho(A))+\hat i(A,\phi(A))+
(\ds(v_0,\rho(A))-\ds(v_0,\phi(A)))]$.
The proof above shows that the point $x_A$ lies at the
index $i(A,x_A)=k_A$ (mod $n\mathbb{Z}$)
if $x_A$ is a vertex,
otherwise $x_A$ lies at the midpoint of
the edge whose endpoints $y,z$
are the vertices with indices $i(A,y),i(A,z)$
given by $k_A \pm \frac{1}{2}$
(mod $n\mathbb{Z}$).

\begin{defn}\label{facetypedefn}
For any face $A$ of $SC(1)$, we define the associated triple
$ft(A) := (\hat i(A,\rho(A)),\hat i(A,\phi(A)),k_A)$.
We define an equivalence relation $\app$ on the
set of faces of $SC(1)$ by $A \app B$ if and only
$ft(A)=ft(B)$; in this case, we say that $A$ and $B$
have the same {\bf face type}.
Define an equivalence relation $\app$ on the set of
vertices of $SC(1)$ by $u \app v$ if and only if
$\Omega(u) \app \Omega(v)$ and
$i(\Omega(u),u) = i(\Omega(v),v)$.
Denote
the equivalence class of a vertex or face $z$ relative to $\app$ by $[z]$.
\end{defn}

Note that
there are only finitely
many face types, and similarly only finitely many $\app$-equivalence
classes of vertices.
For example, it follows from this definition that if $A$
is a face of $SC(1)$  that is attached to $\own(\bs(A))$ at the vertex
$\bs(A)$ in such a way that no edge of $A$ folds onto $\own(\bs(A))$,
then the triple for $A$ is $(n,0,n/2)$, and
$A \app F_1$.  Since $\own(v_0)$ is not a face of $SC(1)$,
the $\app$-equivalence class $[v_0]$ contains only the vertex $v_0$.

The following lemma will be used in the constructions
of a push-down automaton and a finite state automaton later
in this section.

\begin{lemma}\label{facetypelemma}
In $SC(1)$ let $u_1,u_2$ be vertices with $u_1 \app u_2$
and let $e_1=(u_1,x,v_1)$
be an edge.
Suppose that either

\noindent (i) $\own(u_1)=\own(v_1)$,

\noindent (ii) $(\own(u_1),\own(v_1)) \in E(\mathcal{D'})$, or

\noindent (iii)  $(\own(v_1),\own(u_1)) \in E(\mathcal{D'})$ and
$\bar\si(\own(u_1)) \app \bar\si(\own(u_2))$.

\noindent Then there is an edge $e_2=(u_2,x,v_2)$ in $SC(1)$ with
$v_1 \app v_2$ satisfying, respectively,

\noindent (i) $\own(u_2)=\own(v_2)$ and
$\hat i(\own(u_1),v_1)$ lies between  $\hat i(\own(u_1),u_1)$
and $k_{\own(u_1)}$ (inclusive) if and only if
$\hat i(\own(u_2),v_2)$ lies between  $\hat i(\own(u_2),u_2)$
and $k_{\own(u_2)}$.

\noindent (ii) $(\own(u_2),\own(v_2)) \in E(\mathcal{D'})$
and $\bar\si(\own(v_1)) \app \bar\si(\own(v_2))$, or

\noindent (iii)  $(\own(v_2),\own(u_2)) \in E(\mathcal{D'})$.
\end{lemma}

\begin{proof}
Suppose first that $u_1=v_0$.  Then $u_1 \app u_2$ implies that
$u_2=v_0=u_1$, and the result of the lemma follows.  For the
remainder of the proof, we assume that $u_1 \neq v_0$, and
as a consequence $u_2 \neq v_0$.  Let $A_i$ be the face $\own(u_i)$
for $i=1,2$.  By definition of $u_1 \app u_2$, then
$A_1 \app A_2$ and $i(A_1,u_1)=i(A_2,u_2)$.

Suppose that (i) $\own(u_1)=\own(v_1)$ holds.
Then the edge $e_1$ lies in the face $A_1$.
The faces $A_1$ and $A_2$ are copies of the same polygon with
the same boundary label word $w$, and we have $i(A_1,u_1)=i(A_2,u_2)$,
hence there is an edge $e_2=(u_2,x,v_2)$ in the boundary of $A_2$
with $i(A_1,v_1)=i(A_2,v_2)$.
From the definition of $A_1 \app A_2$,
we have  $\hat i(A_1,\phi(A_1))=\hat i(A_2,\phi(A_2))$ and
$\hat i(A_1,\rho(A_1))=\hat i(A_2,\rho(A_2))$.
From Lemma \ref{halflemma}, the edge $e_1$ lies in
$\partial A_1 \setminus \gamma(A_1)$, and so
we have
$\hat i(A_1,\phi(A_1)) < \hat i(A_1,v_1) < \hat i(A_1,\rho(A_1))$.
Then $\hat i(A_2,\phi(A_2)) < \hat i(A_2,v_2) < \hat i(A_2,\rho(A_2))$,
and so $v_2$ lies in $\partial A_2 \setminus \gamma(A_2)$.
Applying the same lemma again gives $\own(v_2)=A_2$.
Then both $v_1 \app v_2$  and the betweenness condition follow directly.

Next suppose that (ii) $(\own(u_1),\own(v_1)) \in E(\mathcal{D'})$.
In this case, $B_1:=\own(v_1)$ is a face of $SC(1)$.
Since $A_1<B_1$, the edge $e_1$ lies in $B_1$,
the vertices $u_1$ and $\bar\si(B_1)$ (which may or may not be the
same point) both lie in $A_1 \cap B_1$, and
$v_1$ lies in $B_1 \setminus \gamma(B_1)$.
Let $B_2$ be the face of $SC(1)$ whose vertex $\bar \si(B_2)$
lies at the vertex of $\partial A_2$ satisfying
$i(A_2,\bar\si(B_2))=i(A_1,\bar\si(B_1))$.  Again using the
fact that the pairs of polygons $A_1,B_1$ and $A_2,B_2$
have the same boundary labels, the gluings of $B_2$ onto
$A_2$ correspond to the gluings of $B_1$ onto $A_1$.
Hence $u_2 \in A_2 \cap B_2$, and there is an edge $(u_2,x,v_2)$
in $B_2$ with $v_2 \notin A_2$.
Since $\own(u_2)=A_2$, we have $A_2<B_2$, and so
$(A_2,B_2) \in E(\mathcal{D'})$.
In addition, we have $\hat i(C_1,\rho(B_1))=\hat i(C_2,\rho(B_2))$ and
$\hat i(C_1,\phi(B_1))=\hat i(C_2,\phi(B_2))$ for $C_i \in \{A_i,B_i\}$ and
$\hat i(B_1,v_1)=\hat i(B_2,v_2)$.  Since
$B_1=\own(v_1)$, then $\hat i(B_1,v_1)$ lies
strictly between $\hat i(B_1,\phi(B_1))$ and $\hat i(B_1,\rho(B_1))$,
and hence $\hat i(B_2,v_2)$ lies strictly between
$\hat i(B_2,\phi(B_2))$ and $\hat i(B_2,\rho(B_2))$,
giving $B_2 = \own(v_2)$.
The property $\bar\si(\own(v_1)) \app \bar\si(\own(v_2))$ follows
immediately.
Now $A_1 \app A_2$ implies that
$k_{A_1}=k_{A_2}$.
Lemma \ref{furthest} shows
that $\ds(v_0,\rho(B_i))=\ds(v_0,x_{A_i})-\ds(x_{A_i},\rho(B_i))$
for $i=1,2$, and similarly for $\phi(B_i)$.  Then
$\ds(v_0,\rho(B_1))-\ds(v_0,\phi(B_1))=
\ds(x_{A_1},\phi(B_1))-\ds(x_{A_1},\rho(B_1))=
|k_{A_1}-\hat i(A_1,\phi(B_1))|-|k_{A_1}-\hat i(A_1,\rho(B_1))|$.
Since all of these numbers are the same if the subscript
1 is replaced by 2 everywhere, then we have
$\ds(v_0,\rho(B_1))-\ds(v_0,\phi(B_1))=\ds(v_0,\rho(B_2))-\ds(v_0,\phi(B_2))$.
This shows that $k_{B_1}=k_{B_2}$, which is the last item
needed to show that $B_1 \app B_2$.  Therefore $v_1 \app v_2$.

Finally, suppose that (iii) $(\own(v_1),\own(u_1)) \in E(\mathcal{D'})$ and
$\bar\si(\own(u_1)) \app \bar\si(\own(u_2))$.
Suppose further that $\bar\si(A_1)=v_0$.
Then $v_1=v_0$ and $A_1=\own(u_1)=F_1$.
In this case $\bar\si(\own(u_2))=v_0$, and so $A_1=A_2$, $u_1=u_2$,
and the lemma holds.

On the other hand, suppose that
$\bar\si(A_1)=\bar\si(\own(u_1)) \neq v_0$.
Then $E_i:=\own(\bar\si(A_i)$ is a face of $SC(1)$ for $i=1,2$,
and we also have $(E_2,A_2) \in \mathcal{D'}$ and $E_1=\own(v_1)$.
The definition of $\bar\si(A_1) \app \bar\si(A_2)$
implies that $E_1 \app E_2$ and
$i(E_1,\bar\si(A_1))=i(E_2,\bar\si(A_2))$.
Now the edge gluings in the folding of $A_1$ onto its parent face
$E_1$ and in the folding of $A_2$ onto $E_2$ must be the same.
The edge $e_1=(u_1,x,v_1)$ lies in $A_1$ with $u_1$ in
$\partial A_1 \setminus \gamma(A_1)$ and $v_1$ in $\gamma(A_1)$,
and there must be a corresponding edge $e_2=(u_2,x,v_2)$
in the face $A_2$.  Then $i(A_1,v_1)=i(A_2,v_2)$, and so
$v_2$ lies in $\gamma(A_2)$.  Hence $E_2=\own(v_2)$.
Finally the correspondence in edge gluings together with
$i(E_1,\bar\si(A_1))=i(E_2,\bar\si(A_2))$ imply
that $i(E_1,v_1)=i(E_2,v_2)$, and so $v_1 \app v_2$.
\end{proof}

Next we use the face type classes of vertices in $SC(1)$
to build a deterministic push-down automaton, following the notation
for a PDA in \cite[p.110]{MR83j:68002}.

\begin{defn}\label{pdadefn}
Let $\pdap=(Q,\Sigma,\Gamma,\delta,q_0,Z_0,F)$ be the
deterministic pushdown
automaton with state set $Q=\set{[v]\suchthat v\in V(SC(1))}$,
input alphabet $\Sigma=\LETTERS$, stack alphabet $\Gamma  =
\set{[v]\suchthat v\in V(SC(1))}$, initial state $q_0=[v_0]$,
initial stack symbol $Z_0=[v_0]$, final (accept) state
$F=\set{[v_0]}$, and transition function the partial function
$\delta\colon Q\times \Sigma \times \Gamma \to Q\times
\Gamma^{*}$ for which $\delta([u],x,[t])$ is defined only
if there is an edge
$(u,x,v)$ for some vertex $v$ in $SC(1)$, by
\[
\delta([u],x,[t]):=\left\{
\begin{array}{ll}
([v],[t]) &
\mbox{if $\Omega(u)=\Omega(v)$} \\
([v],[\sigma(\Omega(v))][t]) &
\mbox{if $(\Omega(u),\Omega(v))\in E(\mathcal{D'})$} \\
([v],\epsilon) &
\mbox{if $(\Omega(v),\Omega(u))\in E(\mathcal{D'})$,
$[t]=[\bar \sigma(\Omega(u))]$}
\end{array} \right. \]
\end{defn}

The undefined transitions for $\delta$ are viewed as
going to a fail state.
Note that for an edge $(u,x,v)$ in $SC(1)$
satisfying $v_0 \neq v=\bar\si(\own(u))=\rho(\own(u))=\phi(\own(u))$, so
that the face $\own(u)$ is attached at $v$ but no edges are glued,
the last case of the definition of $\delta$ can
be split into two subcases.
In this situation we have $\own(u) \app F_1$, and there is an edge $(u_1,x,v_0)$
in $F_1$ with $u_1 \app u$.  If $[t]=[v] \neq [v_0]$, then
$\delta([u],x,[t]):=([v],\epsilon)$, but if $[t]=[v_0]$, then
$\delta([u],x,[t]):=([v_0],\epsilon)$.
The fact that $\delta$ is well-defined
follows directly from Lemma \ref{facetypelemma}.

An instantaneous description $(\alpha,z,\beta)$ for the PDA
$\pdap$ consists of the current state $\alpha \in Q$ of the
machine, the word $z \in (X \cup X^{-1})^*$ that remains to be
read, and the current contents $\beta \in \Gamma^*$ of the stack,
where the first letter of $\beta$ is the ``top'' of the stack. We
write $(\alpha,yz,\beta) \vdash^* (\alpha',z,\beta')$ if, when $y$
is read in starting from $(\alpha,yz,\beta)$, the PDA reaches
$(\alpha',z,\beta')$, and write $\vdash$ when a single letter $y
\in X \cup X^{-1}$ is read.

Define a function $\beta:V(SC(1)) \rightarrow \Gamma^*$ as
follows.
Given any vertex $v$ in $SC(1)$, let $v_0,F_1,...,F_m=\own(v)$ be
the labels of the vertices along the geodesic path
in the tree $\mathcal{D'}$ from $v_0$ to $\own(v)$.
Then $\beta(v)$ is the associated word over the stack alphabet
given by
$\beta(v):=[\bar\si(F_m)] \cdots [\bar\si(F_1)][v_0]$.

\begin{prop}\label{pdathm}
Let $w$ be sparse, and let $SC(1)$ be the \sch\ complex of 1 for
$M=\Inv{X}{w=1}$.  Let $\alpha \in Q$, $y,z \in (X \cup X^{-1})^*$,
and $\beta \in \Gamma^*$.  Then
$([v_0],yz,[v_0])\vdash^* (\alpha,z,\beta)$ if and only if
$y$ labels an edge path in $SC(1)$ starting at $v_0$ and
$\alpha=[v]$ and $\beta=\beta(v)$ where $v$ is the end
vertex of this path.
\end{prop}

\begin{proof}
First we prove the forward implication, by induction on the
length of $y$.  If $l(y)=0$,
then $y=\epsilon$
and $([v_0],\epsilon z,[v_0])\vdash^* (\alpha,z,\beta)$
implies that $\alpha=[v_0]$, and $\beta=[v_0]=\beta(v_0)$.
The path starting at $v_0$ labeled by $y=\epsilon$ ends at $v=v_0$,
as required.

Now, suppose that the forward implication holds for any word
$\tilde y$ with $0 \le l(\tilde y)<l(y)$, and write
$y=y'x$ with $x \in \LETTERS$.
Suppose that $([v_0],yz,[v_0]) \vdash^* (\alpha,z,\beta)$.
Then we have
$([v_0],y'xz,[v_0])\vdash^*(\alpha',xz,\beta')\vdash(\alpha,z,\beta)$
for some $\alpha' \in Q$ and $\beta'\in \Gamma^*$.
By induction, the word $y'$ labels a path $\pi'$ in $SC(1)$ starting
at $v_0$, and $\alpha'=[u]$ and $\beta'=\beta(u)$ where $u$ is the
ending vertex of the path $\pi'$.

Since $(\alpha',xz,\beta')\vdash(\alpha,z,\beta)$, the
transition function $\delta$ is defined on the triple
$(\alpha',x,\gamma)$, where $\gamma$ is the
first letter of the word $\beta(u) \in \Gamma^*$.
This means that there is a representative $\tilde u$
of the $\app$-class $\alpha'$ such that there is an edge of the form
$e=(\tilde u,x,v)$ in $SC(1)$ for some vertex $v$, and either
(i) $\own(\tilde u)=\own(v)$,
(ii) $(\own(\tilde u),\own(v)) \in E(\mathcal{D'})$, or
(iii) $\gamma=[\bar\si(\own(\tilde u))]$
and $(\own(v),\own(\tilde u))  \in E(\mathcal{D'})$.
In cases (i) and (ii), Lemma \ref{facetypelemma} shows that we
may take $\tilde u=u$.  In case (iii),
notice that the first letter $\gamma$ of $\beta(u)$
satisfies $\gamma=[\bar\si(\own(u))]$ if $\own(u) \neq v_0$, and
$\gamma=[v_0]$ if $\own(u)=v_0$.  However, if $\own(u)=v_0$, then
$u=v_0$, and since $[\tilde u]=\alpha'=[u]$, then $\tilde u=v_0$,
contradicting the existence of the edge
$(\own(v),\own(\tilde u))$ in $\mathcal{D'}$.
Then $\own(u) \neq v_0$, and so we also
may take $\tilde u=u$ in this case.

Then in all three cases, the path $\pi'$ followed by
the edge $e$ is a path in $SC(1)$ labeled by the
word $y$ starting at $v_0$ and ending at the vertex $v$.
Moreover, we have $\alpha=[v]$.

In case (i), $\delta(\alpha',x,\gamma)=([v],\gamma)$,
and the stack word $\beta=\beta'=\beta(u)$ is unchanged by
this transition.  Since $\own(u)=\own(v)$, then $\beta=\beta(v)$.

In case (ii), $\delta(\alpha',x,\gamma)=([v],[\bar\si(\own(v))]\gamma)$,
and we have $\beta=[\bar\si(\own(v))]\beta(u)$.  Since
$(\own(u),\own(v)) \in E(\mathcal{D'})$, we again
have $\beta=\beta(v)$.

In case (iii), $\delta(\alpha',x,\gamma)=([v],\epsilon)$.
Now $(\own(v),\own(u))  \in E(\mathcal{D'})$ implies that
$\beta(u)=[\bar\si(\own(u))]\beta(v)$, and we have $\beta=\beta(v)$
in this case as well.

This completes
the proof of the forward implication.

For the reverse implication, we again induct on the length $l(y)$.
If $l(y)=0$, the as before $y=\epsilon$ labels a path from $v_0$
to $v_0$, and so $([v_0],yz,[v_0]) \vdash^* (\alpha,z,\beta)$
where $\alpha=[v_0]$ and $\beta=[v_0]=\beta(v_0)$.

Suppose again that $l(y)>0$ and write $y=y'x$ with $x \in \LETTERS$.
By hypothesis, $y$ labels a path in $SC(1)$ from $v_0$; let $v$
be the vertex at the end of this path, and let $u$ be the penultimate
vertex; that is, $u$ is at the end of the path labeled by $y'$.
By induction we have $([v_0],yz,[v_0]) \vdash^* ([u],xz,\beta(u))$.
The definition of $\delta$ then shows that
$([v_0],yz,[v_0]) \vdash^* ([v],z,\beta(v))$.
\end{proof}

We can now prove Theorems \ref{cfree} and \ref{geodesics}.

\medskip

\noindent {\bf Proof of Theorem \ref{cfree}}

For a word $y \in \WORDS$, we have
 $y=1$ in $M$ if and only if $y$ labels an edge path
from $v_0$ to $v_0$ in $S\Gamma(1)$.
Proposition \ref{pdathm} shows that the latter holds if and only if
$([v_0],y,[v_0])\vdash^* ([v_0],\epsilon,\beta)$ for some $\beta$; that
is, exactly when the PDA $\pdap$ finishes in the accept state $[v_0]$.
Thus, the set of words representing the
identity element in $M$ is a deterministic context-free language.

The word $y$ is in the
language of words related to $1$ in $M$ by Green's relation $\gR$
if and only if $y$ labels a path starting at $v_0$ in $S\Gamma(1)$, which
holds if and only if
$([v_0],y,[v_0])\vdash^* (\alpha,\epsilon,\beta)$ for some $\alpha\in Q$ and
$\beta\in \Gamma^*$.  Let $\pdap'$ be the PDA $\pdap$ with the
the set of final (accept)
states changed to $F=Q$.  Then we have $y$ is accepted by $\pdap'$
if and only if $y$ is in the $\gR$-equivalence class $\gR_1$ of 1.  Hence
the set of words representing an element
of $\gR_1$ in $M$ is also a deterministic context-free
language.
\hfill $\Box$

\medskip

\noindent {\bf Proof of Theorem \ref{geodesics}}

Let $(Q,\Sigma,\delta,q_0,F)$ be the finite state automaton with
state set $Q=\set{[v]\suchthat v\in V(SC(1))}$, input alphabet
$\Sigma=\LETTERS$, initial state $q_0 = [v_0]$, final (accept)
states $F = Q$, and transition function
 the partial function $\delta\colon Q\times \Sigma \to
Q$,  defined by $\delta([u],x):=[v]$ if there is an edge
$(u,x,v)$ in $SC(1)$ and either

\noindent (i) $\Omega(u)=\Omega(v)$ and either
$\hat i(\Omega(u),u) < \hat i(\Omega(u),v) \leq k_{\own(u)}$ or
$\hat i(\Omega(u),u) >
\hat i(\Omega(u),v) \geq k_{\own(u)}$, or

\noindent (ii) $(\Omega(u),\Omega(v))\in E(\mathcal{D'})$

\noindent Lemma \ref{facetypelemma} shows that
this transition function is well-defined.

Let $p$ be an arbitrary path in $SC(1)$ starting at $v_0$.  Let
$v_0,v_1,...,v_m$ be the sequence of consecutive vertices
traversed by $p$, and let $A_i:=\own(v_i)$.
Note that the path $p$ is geodesic if and only if
$\ds(v_0,v_{i-1})>\ds(v_0,v_i)$ for all $i$.

If $(A_i,A_{i-1})\in E(\mathcal{D'})$, then
the Geodesic Theorem \ref{geodesicprop} says that
$p$ is not a geodesic.
If $(A_{i-1},A_i)\in E(\mathcal{D'})$, then
$v_{i} \in \partial A_i \setminus \gamma(A_i)=\partial A_i \setminus A_{i-1}$,
and the vertex $v_{i-1}$ must be one of the endpoints
$\rho(A_i),\phi(A_i)$ of the gluing path
of $A_i$ onto $A_{i-1}$; let $u_i$ be the other.
The Geodesic Theorem \ref{geodesicprop}
says that any geodesic from $v_0$ to $v_i$ must
pass through one of the points $v_{i-1},u_i$.  Since
$\ds(v_{i-1},v_i)=1$, then Lemma \ref{halflemma}(1) shows
that such a geodesic must also pass through $v_{i-1}$.
Hence $\ds(v_0,v_{i})>\ds(v_0,v_{i-1})$.
Finally, if $A_{i-1}=A_i$, then $v_{i-1}$
and $v_i$ are both vertices in $\partial A_i \setminus \gamma(A_i)$.
By Lemma \ref{furthest},
it follows that $\ds(v_0,v_i)>\ds(v_0,v_{i-1})$ if and only if either
$i(A_i,v_{i-1})< i(A_i,v_i) \leq k_{A_i}$ or
$i(A_i,v_{i-1})> i(A_i,v_i) \geq k_{A_i}$.

In the proof of Theorem \ref{cfree}, we showed that
a word $y$ labels a path starting at $v_0$ in $SC(1)$ if and only if it
is accepted by the PDA $\pdap'$, which is the PDA in Definition
\ref{pdadefn} but for which all states in $Q$ are final (accept)
states.  Note that the only transitions of this PDA which
utilize the stack in determining the next state are those
associated with edges from $u$ to $v$ with
$(\own(v),\own(u)) \in E(\mathcal{D'})$.  Combining this with
the previous paragraph, then,
the finite state automaton defined above
is precisely the underlying finite state
automaton of the PDA $\pdap'$ consisting only of transitions associated
with edges $(u,x,v)$ such that $d(v_0,v)>d(v_0,u)$. Thus this
finite state automaton accepts precisely the words which label
geodesic paths in $SC(1)$.  \hfill $\Box$
\medskip

{\bf Remark 1.} The minimized form of the finite state automaton
defined in the proof of Theorem \ref{geodesics} is the automaton
of cone types of $S\Gamma(1)$. As an example, S. Haataja showed
that the automaton of cone types for $S\Gamma(1)$ for the sparse word $w
= aba^{-1}b^{-1}cdc^{-1}d^{-1}$ corresponding to the surface group
of genus $2$ has $19$ cone types (unpublished manuscript). A
description of Haataja's example may be found in Meakin's survey
article \cite{Me07}.

{\bf Remark 2.} Descriptions of an
iterative construction of the PDA in
Definition \ref{pdadefn} and an  implementation of the algorithm
for solving the word problem is provided in S. Lindblad's PhD
thesis \cite{lindphd}.  The software is available from
http://www.math.unl.edu/~shermiller2/lindblad/ .

{\bf Remark 3.} In their paper \cite{MR2001m:20107}, Ivanov,
Margolis and Meakin show that the word problem for the inverse
monoid $M=\Inv{X}{w=1}$ corresponding to a cyclically reduced word
$w$ is solvable if the membership problem for the submonoid
 of the corresponding one-relator group $G = Gp \langle X | w = 1 \rangle$
generated by the prefixes of $w$ is solvable. However as far as we
are aware, it is not known whether the prefix membership problem
for this submonoid of $G$ is equivalent to the word problem for
$M$ in general. In particular, it is not known whether this prefix
membership problem for $G$ is solvable if $w$ is a sparse word.

\newcommand{\etalchar}[1]{$^{#1}$}
\providecommand{\bysame}{\leavevmode\hbox to3em{\hrulefill}\thinspace}
\providecommand{\MR}{\relax\ifhmode\unskip\space\fi MR }
\providecommand{\MRhref}[2]{%
  \href{http://www.ams.org/mathscinet-getitem?mr=#1}{#2}
}
\providecommand{\href}[2]{#2}

\end{document}